\documentclass{amsart}
\usepackage{amsmath,amssymb,enumerate}
 \usepackage[applemac]{inputenc} % (pour les accents)

\usepackage{epsfig,fancyhdr,color}%,showkeys,amsmidx
\usepackage{psfrag}
\usepackage{graphicx} 

\newcommand{\tq}{\, \big| \, }
\newtheorem{theorem}{\rm\bf Theorem}
\newtheorem{proposition}{\rm\bf Proposition}[section]
\newtheorem{lemma}[proposition]{\rm\bf Lemma}

\theoremstyle{definition}
\newtheorem{definition}[proposition]{\rm\bf Definition}

\theoremstyle{remark}
\newtheorem{remark}[proposition]{\rm\bf Remark}

\newtheorem{example}[proposition]{\rm\bf Example}

\def\interieur#1{\mathord{\mathop{\kern 0pt #1}\limits^\circ}}

\title [Harmonic symmetrization of convex sets]
{Harmonic symmetrization of convex sets and of Finsler structures, with applications to Hilbert geometry}
\author{Athanase Papadopoulos}
\address{A. Papadopoulos, Institut de Recherche Math{\'e}matique Avanc\'ee,
Universit{\'e} Louis Pasteur and CNRS,
7 rue Ren\'e Descartes,
 67084 Strasbourg Cedex - France} \email{papadopoulos@math.u-strasbg.fr}

\author{Marc Troyanov}
\address{M. Troyanov, Section de Math{\'e}matiques,  \'Ecole Polytechnique F{\'e}d\'erale de
Lausanne, 1015 Lausanne - Switzerland}
\email{marc.troyanov@epfl.ch}

\date{September 15, 2008.}

\begin{document}

\begin{abstract}  

David Hilbert discovered in 1895 an important metric that is canonically associated to an arbitrary convex domain $\Omega$ in the Euclidean (or projective) space. This metric is known to be Finslerian, and the usual proof of this fact assumes a certain degree of smoothness
of the boundary of $\Omega$, and refers to a theorem by Busemann and Mayer that produces the norm of a tangent vector from the distance function. In this paper, we develop a new approach for the study of the Hilbert metric where  no differentiability is assumed. The approach exhibits the Hilbert metric on a domain as a  symmetrization of a natural weak metric, known as the Funk metric. The Funk metric is described as a \emph{tautological} weak Finsler metric, in which the unit ball in each tangent space is naturally identified with the domain $\Omega$ itself.  The Hilbert metric is then identified with  the \emph{reversible tautological weak Finsler structure} on $\Omega$, and the unit ball of the Hilbert metric at each point is described as the \emph{harmonic} symmetrization of the unit ball of the Funk metric. Properties of the Hilbert metric then follow from general properties of harmonic symmetrizations of weak Finsler structures.
   
   \medskip
   

\end{abstract}

\maketitle
\tableofcontents
 
\section{Introduction}\label{intro}

The Hilbert metric is a canonical metric associated to an arbitrary bounded convex domain $\Omega \subset \mathbb{R}^n$. It has been proposed by David Hilbert in 1895  as an example of a metric for which the Euclidean straight lines are shortest geodesic curves. In the special case where $\Omega$ is the unit ball $\mathbb{B}^n  \subset \mathbb{R}^n$, this metric had been previously introduced by Felix Klein as a model of the hyperbolic (Lobachevski) space. The Hilbert metric has been very actively studied in recent years under  various  viewpoints by several authors, see in particular the papers by   Colbois, Verovic and Vernicos \cite{CV}, \cite{Vernicos2005}, \cite{CVV},  F\"ortsch,   Karlsson and  Noskov \cite{FK}, \cite{KN}, de la Harpe \cite{Harpe},   Benoist \cite{Benoist1}, \cite{Benoist2}, \cite{Benoist3}, \cite{Benoist4},  the thesis of Soci\'e-M\'ethou \cite{Socie1},\cite{Socie2}  and the book by Chern and Shen \cite{ChernShen}.

\medskip

To state things more precisely, we briefly recall the definition of the Hilbert metric. Consider two distinct points $x$ and $y$ in the bounded convex domain $\Omega$. The Euclidean line through $x$ and $y$ intersects the boundary of $\Omega$ at two points, which we denote by $a^+$ and $a^-$, in such a way that
$a^-,x,y,a^+$ are aligned in that order.  

\begin{figure}[htp]
\centering
\includegraphics[height=4cm]{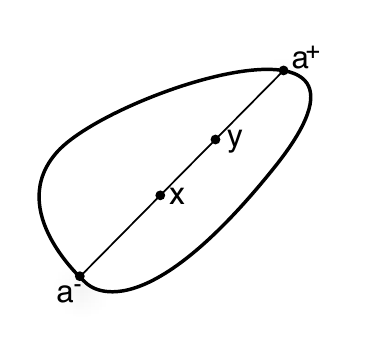}  
\caption{\small{The points $a^-,x,y,a^+$ used to define the Hilbert distance between $x$ and $y$.}}\label{fig1}
\end{figure}

The Hilbert metric $H$ is then defined by the formula
\begin{equation}\label{hilbert}
  H(x,y)= \frac{1}{2} \log\left( {\vert x-a^+\vert\over \vert y-a^+\vert}  {\vert y-a^-\vert\over \vert 
x-a^-\vert}\right).
\end{equation}

The following three basic facts are well known to people familiar with the Hilbert metric:
\begin{enumerate}[i)]
  \item The formula (\ref{hilbert})  is indeed a metric.
  \item This metric is Finslerian, provided the boundary of $\Omega$ is smooth enough.
  \item The metric is projective, that is, the Euclidean straight lines are geodesic.
\end{enumerate}
These facts are somewhat delicate to prove (the triangle inequality is not so simple to check, see e.g. \cite{Hilbert, Hilbert2}).  The main goal of the present paper is to give a new point of view on these facts and to provide simple proofs of them. We also extend the second property to any convex set, getting rid of any smoothness condition.

\medskip

To say that the Hilbert metric is \emph{Finslerian} means that the distance between two points is the infimum of the length of all (piecewise smooth) curves joining these two points, the length of a curve  $\gamma : [a,b] \to \Omega$ being defined as 
\begin{equation} \label{eq.lengthpv}
  \ell (\gamma) = \int_a^b p(\gamma(t),\dot\gamma(t)) dt.
\end{equation}

Here $p$ is a continuous function associated to the Finsler structure, which is defined on the tangent bundle of the domain $\Omega$, whose restriction to every fiber is a norm and which is smooth in the complement of the zero section. This function is called the \emph{Lagrangian} of the Finsler structure.

\medskip

The usual proof that the Hilbert metric of a smooth convex domain is Finslerian is quite involved. In this proof, one starts from the Busemann-Mayer Theorem \cite{BusemannMayer} which gives the Lagrangian of a Finsler structure as an infinitesimal version of the distance. In the case of the Hilbert metric, this theorem says that
$$
 p(x,\xi) = \lim_{t\to 0}\frac{H(x,\gamma(t))}{t},
$$
where $\gamma$ is any $C^1$ curve such that $\gamma(0) = x$ and $\dot\gamma(t)=\xi$. A  calculation gives then
$$
  p(x,\xi) = \frac{|\xi|}{2}\left(\frac{1}{|x-a|} + \frac{1}{|x-b|} \right),
$$
where $a$ and $b$ are the intersection points of the line $L$ through $x$ in direction $y$ with $\partial \Omega$. One then computes the length of a segment joining two points $x$ and $y$ using Formula (\ref{eq.lengthpv}), and one finds that this length is equal to $H(x,y)$. Finally, one proves that the length of any smooth curve joining $x$ to $y$  does not exceed $H(x,y)$. This is done by a delicate argument where the length of a smooth curve is approximated by that of a polygonal curve. The proof is sketched in \cite{Vernicos2005} and given with more details in \cite{Socie1}.

\bigskip

In the present paper, we approach the Hilbert metric from another point of view, in which this metric appears as a natural \emph{reversible  tautological} weak Finsler structure associated to the convex set $\Omega$.   In the spirit of our previous papers \cite{PT1,PT2}, we first deal with a simpler non-symmetric version of the Hilbert metric (called the Funk weak metric), which appears as the \emph{ tautological} weak Finsler structure on $\Omega$,  and we then symmetrize that metric.
Our approach has the advantage of making no smoothness assumptions, and no reference to the delicate Busemann-Mayer Theorem. As a side benefit, the proof of the triangle inequality of the Hilbert metric comes for free.

  The results of this paper can be considered as a continuation of a program that we started in \cite{PT1}, in which we investigate non-symmetric distances and their applications.

We would like to thank the referee for pointing out a number of inaccuracies and mistakes in the original manuscript.

\section{Weak metrics and their symmetrization}\label{sec:weak}

 \begin{definition}
 A \emph{weak metric} on a set $X$ is a function $\delta : X\times X\to [0,\infty]$ satisfying
 \begin{enumerate}
 \item $\delta(x,x)=0$ for all $x$ in $X$;
 \item $\delta(x,z)\leq \delta(x,y)+\delta(y,z)$ for all $x$, $y$ and $z$ in $X$.
 \end{enumerate}
 \end{definition}
 
The  weak metric $\delta$ is said to be  \emph{symmetric} if $\delta(x,y)=\delta(y,x)$ for all $x$ and $y$ in $X$, it is said to be \emph{finite} if $\delta(x,y)<\infty$ for every $x$ and $y$ in $X$, and it is said to be \emph{strongly separating} if we have the equivalence
\[\min(\delta(x,y),\delta(y,x))=0\iff x=y.\]
Finally, the weak metric $\delta$ is said to be \emph{weakly separating} if we have the equivalence
\[\max(\delta(x,y),\delta(y,x))=0\iff x=y.\]

The notion of weak metric goes back to the first half of the last century  (see e.g. \cite{Hausdorff}, in which Hausdorff defines asymmetric distances on various sets of subsets of a metric space). Asymmetric metrics were extensively studied by Busemann, cf. \cite{Busemann1942},  \cite{Busemann1944}, \cite{Busemann1955} \& \cite{Busemann1970}. 

A simple example of a weak metric is the   Minkowski weak metric discussed in the next section, and additional examples are given in the paper \cite{PT1}. An example that plays a fundamental role in the present paper is  the Funk weak metric, which is defined as follows:

\begin{definition}[The Funk weak metric]\label{def:Funk}
Let $\Omega$ be a nonempty open
convex subset of $\mathbb{R}^n$. The \emph{ Funk weak metric} of $\Omega$, denoted by $F=F_{\Omega}$,  is the weak metric defined, for $x$ and $y$ in $\Omega$, by the formula
\[
  \displaystyle F(x,y)=
\begin{cases} \displaystyle \log \frac{\vert x-a^+\vert}{\vert y-a^+\vert} & \text{ if } x\not= y \text{ and } R(x,y)\not\subset \Omega\\
0 & \text{ otherwise}.
\end{cases}
\]
\end{definition}
In this definition,   $R(x,y) \subset \mathbb{R}^n$ is the ray (i.e. the half-line) with origin $x$ and passing through the point $y$ and $a^+ = R(x,y)\cap \partial \Omega$.
The geometry of the Funk weak metric is discussed in \cite{PT2} and \cite{Zaustinsky}. Note that the classical proof 
of the triangle inequality for the Funk weak metric is based on a nonobvious geometric argument (see \cite{Zaustinsky}), but in our approach, we prove that the Funk weak metric is weak Finslerian and the triangle inequality comes for free. We shall come back 
on this at the end of the  section \ref{TWF}.

\medskip
 
 There are several ways to associate a symmetric weak metric to a given weak metric, and we shall use the symmetrization ${}^s\delta$ of $\delta$ defined by the formula
\begin{equation}\label{star}
{}^s\delta(x,y)=\frac{1}{2}\left(\delta(x,y)+\delta(y,x)\right)       
\end{equation}
for $x$ and $y$ in $X$. We shall call  ${}^s\delta$ the \emph{arithmetic symmetrization} of $\delta$.

Although we shall not use this fact in this paper, we note that there are other possible ways to symmetrize a given weak  metric. An example is the \emph{max symmetrization}, defined as 
\[M\delta(x,y)= \max\{\delta(x,y),\delta(y,x)\}\] 
for $x$ and $y$ in $X$.

\section{The Minkowski weak metric}\label{sec:Min}

For $n\geq 0$, let $\Omega\subset\mathbb{R}^n$ be a convex set such that $0\in\overline{\Omega}$ (the closure of $\Omega$),  and let $p:\mathbb{R}^n\to [0,\infty]$ be the function defined by

\[p(\xi)=\inf\{t>0\ \vert \ \frac{1}{t}\xi\in \Omega\}.\]

Note that if the ray $\mathbb{R}_+ \xi$ intersects the boundary $\partial\Omega$, say at a point $a$, then 
\[p(\xi)= \frac{|\xi|}{|a|},\]
otherwise $p(\xi) = 0$.
The function $p$ is called a \emph{Minkowski weak norm}. Minkowski weak norms (sometimes under different names) are studied in various books,  e.g. \cite{Eggleston},  \cite{Minkowski},  \cite{Thompson} and  \cite{Webster}.

The function $\delta:\mathbb{R}^n\times \mathbb{R}^n\to  [0,\infty]$ defined by 
\begin{equation}\label{eq:w}  
\delta(x,y)=p(y-x)
\end{equation}
is a weak metric on $\mathbb{R}^n$. We have the following  relations between 
the properties of the weak metric and the convex set $\Omega$:
\begin{enumerate}
\item $\delta$ is finite $\iff$ $0\in \interieur{\Omega}$ (the interior of $\Omega$);
\item if $\Omega=-\Omega$, then $\delta$ is symmetric;
\item $\delta$ is strongly separating $\iff$ $\Omega$ does not contain any Euclidean ray;
\item $\delta$ is weakly separating $\iff$ $\Omega$ does not contain any Euclidean line.
\end{enumerate}

We shall return to Minkowski weak metrics in \S\ref{s:sym} below, where we shall investigate their symmetrization. In particular, we shall construct, for each convex set $\Omega$, a symmetric convex set $\mathcal{H}(\Omega)$ whose associated weak Minkowski metric is the arithmetic symmetrization of the weak metric (\ref{eq:w}) associated to $\Omega$.

%__________________
\section{Weak length spaces and their symmetrization}

Let $X$ be a topological space. We shall say that a collection $\Gamma$  of continuous paths $\gamma: [a,b]\to X$, where $[a,b]$ can be any compact interval of $\mathbb{R}$, is a \emph{semigroupoid of paths} on $X$ if  the following properties hold:
\begin{enumerate}
\item if $\gamma_1:[a,b]\to X$ and $\gamma_2:[c,d]\to X$ satisfy $\gamma_1(b)=\gamma_2(c)$, then the concatenation $\gamma_1 *\gamma_2$ is in $\Gamma$, 
\item any constant path belongs to $\Gamma$.
\end{enumerate}
 
 \smallskip
 
A typical example of a semigroupoid of paths is given by the set  of all  piecewise smooth paths in a smooth manifold.

\smallskip

\textbf{Remarks.}   In reference to the abstract notion of semigroupoid, it would not be necessary to assume that all constant paths belong to $\Gamma$, but this hypothesis is convenient and does not reduce the generality of our concepts. 

\medskip

We shall use the following notion:

\begin{definition}[weak length structure] 
Let $X$ be a topological space and let $\Gamma$ be a semigroupoid of paths on $X$.
A \emph{weak length structure} on $(X,\Gamma)$ is a function $\ell:\Gamma\to [0,\infty]$ 
such that the following two properties are satisfied:
\begin{enumerate}
\item (Additivity.) For every  $\gamma_1$ and $\gamma_2$ in $\Gamma$, we have $\ell(\gamma_1 *\gamma_2)=\ell(\gamma_1) + \ell(\gamma_2)$.
\item For any constant path $c$, we have $\ell(c)=0$.
\item (Invariance under reparametrization.) If $[a,b]$ and $[c,d]$ are intervals of $\mathbb{R}$, if $\gamma:[a,b]\to X$ is a path in $X$ which is in $\Gamma$ and if $f:[c,d]\to [a,b]$ is a continuous surjective nondecreasing map such that $\gamma\circ f$ is in $\Gamma$, then $\ell(\gamma)=\ell(\gamma\circ f)$.
 \end{enumerate}
\end{definition}

\medskip

\begin{definition}
 A \emph{weak length space} is a triple $(X,\Gamma,\ell)$ where $X$ is a topological space, $\Gamma$ is a semigroupoid of paths on $X$ and $\ell$ is weak length structure on $(X,\Gamma)$.
\end{definition}

\medskip

Let us give a few additional definitions:

\smallskip

$ \bullet$ \  The  weak length structure $\Gamma$ is   \emph{separating} if  
$\ell(\gamma)>0$ for any non constant path $\gamma$ in $\Gamma$.

\smallskip

$ \bullet$ \  The weak length structure $\Gamma$ is said to be \emph{reversible} if for every $\gamma$ in $\Gamma$ we have 
$\gamma^{-1}\in \Gamma$ and $\ell(\gamma^{-1})=\ell(\gamma)$, where $\gamma^{-1}$ is the reverse path of $\gamma$.

\smallskip

$ \bullet$ \   Let $(X,\Gamma,\ell)$ be a weak length space such that $\gamma^{-1}\in \Gamma$ for every $\gamma$ in $\Gamma$. Then one defines the  \emph{arithmetic symmetrization} of the weak length structure $\ell$  to be the weak length structure ${}^s\ell$ on $(X,\Gamma)$ given by
\[
 {}^s\ell(\gamma)=\frac{1}{2}\left(\ell(\gamma^{-1})+\ell(\gamma)\right).
 \]

\smallskip

Given a groupoid of paths $\Gamma$ on a topological space $X$, for  $x$ and $y$ in $X$, we let   \[\Gamma_{x,y}=\{\gamma\in \Gamma\ \vert \ \gamma \textrm{ \emph{joins $x$ to $y$}} \}.\]

\begin{lemma}
Let $(X,\Gamma,\ell)$ be a topological space equipped with a semigroupoid of paths and with a weak length structure.
Then the function $\delta_{\ell} : X\times X \to \mathbb{R}$ defined by
\begin{equation} \label{weakm}
 \displaystyle \delta_{\ell}(x,y)=\inf_{\gamma\in\Gamma_{x,y}} \ell(\gamma),
\end{equation}
is a weak metric on $X$. This weak metric is symmetric if $\ell$ is symmetric.
If $\delta_{\ell}$ is separating, then $\ell$ is separating.
\end{lemma}

The proof is immediate from the definitions.
\qed

\medskip

\begin{definition} Let $(X,\Gamma,\ell)$ be a topological space equipped with a semigroupoid of paths and with a weak length structure. The weak metric $\delta_{\ell}$ defined in (\ref{weakm}) is called the \emph{weak metric associated to the weak length structure $\ell$}. A \emph{weak length metric space} is a weak metric space obtained from the triple $(X,\Gamma,\ell)$ by equipping $X$ with the associated weak metric $\delta_{\ell}$. 
\end{definition}

Given a weak length structure $\ell$ on a pair $(X,\Gamma)$ as above, we can consider, on the one hand, the associated weak metric $\delta_{\ell}$ and then its arithmetic symmetrization ${}^s\delta_{\ell}$, and on the other hand, the arithmetic symmetrization ${}^s\ell$ and the resulting weak metric $\delta_{{}^s\ell}$. The two functions ${}^s\delta_{\ell}$ and $\delta_{{}^s\ell}$ defined on $X\times X$ are not necessarily equal, but there is an inequality that is always satisfied, as it is shown in the following:

\begin{lemma}\label{lem:eq} Let $(X,\Gamma,\ell, \delta_{\ell})$ be a  weak length metric space. Then, we have, for every $x$ and $y$ in $X$,
\[\delta_{{}^s\ell}(x,y) \geq {}^s\delta_{\ell}(x,y).\]
In general, we do not have equality.
\end{lemma}

\begin{proof}
For every $\epsilon >0$, we can find an element $\gamma$ in $\Gamma_{x,y}$ satisfying

\begin{eqnarray*}
\delta_{{}^s\ell}(x,y)&\geq&{}^s\ell(\gamma)-\epsilon\\
&=& \frac{1}{2}\left(\ell(\gamma) +\ell(\gamma^{-1})\right)-\epsilon\\
&\geq&
 \frac{1}{2}\left(\delta_{\ell}(x,y) + \delta_{\ell}(y,x))\right)-\epsilon
 \\ &=&
  {}^s\delta_{\ell}(x,y) -\epsilon.
 \end{eqnarray*}
Since $\epsilon > 0$ is arbitrary, we obtain the required result.
\end{proof}
An example where equality fails in the above Lemma \ref{lem:eq} is the following:

\begin{example}\label{ex:ex} Let $X$ be homeomorphic to the circle, equipped with the semigroupoid of all piecewise smooth paths, and let $x$ and $y$ be two distinct points in $X$. Up to reparametrization, there are exactly two injective paths in $X$ joining $x$ to $y$, and we call them respectively the ``upper path" and the ``lower path". Likewise, there are  two injective paths from $y$ to $x$ (with the same adjectives). We can easily put a weak length space structure $\ell$ on  $X$ such that the following properties hold:
\begin{itemize} 
 \item the length of the upper path from $x$ to $y$ is equal to 9;
  \item the length of the lower path from $x$ to $y$ is equal to 1;
  \item the length of the upper path from $y$ to $x$ is equal to 1;
  \item the length of the lower path from $y$ to $x$ is equal to 9.
\end{itemize}

\vspace{-.35cm}
\begin{figure}[htp]
\centering
 \includegraphics[height=4cm]{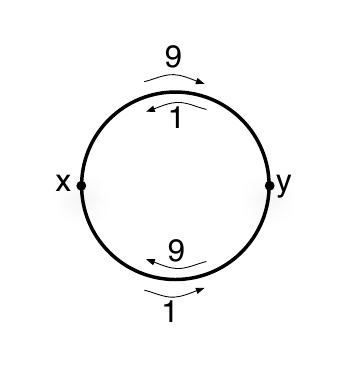}
 \caption{\small{The length space used in Example \ref{ex:ex}.}}\label{fig1}
\end{figure}
\vspace{-.1cm}

With these conditions, the associated distances satisfy
$\delta_{\ell}(x,y)=1=\delta_{\ell}(y,x)$.

Now consider the symmetrization ${}^s\ell$ of the length function, and the associated 
distance function $\delta_{{}^{s}\ell}$. The  ${}^s\ell$-length of the four injective paths 
considered above is equal to 5, and we have 
$ \delta_{{}^{s}\ell}(x,y)=5$,
 which is not equal to the arithmetic mean of 
$\delta_{\ell}(x,y)$ and $\delta_{\ell}(y,x)$.
 \end{example}

There is an important instance where equality holds in Lemma \ref{lem:eq}, and to state it we make the following definition:

\begin{definition}[Minimal and bi-minimal paths]
Let $(X,\Gamma,\ell, \delta_{\ell})$ be a  weak length metric space. A path $\gamma\in\Gamma_{x,y}$ is said to be \emph{minimal} if $\gamma^{-1}\in \Gamma_{y,x}$ and if $\ell(\gamma)=\delta_{\ell}(x,y)$. The path $\gamma$ is said to be \emph{bi-minimal} if $\gamma$ and $\gamma^{-1}$ are minimal, that is, if $\gamma^{-1}\in \Gamma_{y,x}$, $\ell(\gamma)=\delta_{\ell}(x,y)$ and $\ell(\gamma^{-1})=\delta_{\ell}(y,x)$.
\end{definition}

The following proposition will be useful:
\begin{proposition}\label{prop:minimizing}
 Let $(X,\Gamma,\ell, \delta_{\ell})$ be a weak length metric space and let  $x$ and $y$ be two points in $X$ such that there exists a bi-minimal path  $\gamma\in\Gamma_{x,y}$. Then, we have
 \[{}^s\delta_{\ell}(x,y)= \delta_{{}^s\ell}(x,y).\]

\end{proposition} 
\begin{proof}
Let $\gamma$ be a bi-minimal path from $x$ to $y$. Then,   
\begin{eqnarray*}
{}^s\delta_{\ell}(x,y)&=&\frac{1}{2}\left(\delta_{\ell}(x,y) + \delta_{\ell}(y,x) \right)\\
&=& \frac{1}{2}\left(\ell(\gamma) +\ell(\gamma^{-1})\right)\\&=&
\frac{1}{2}\inf \{\ell(\alpha)+\ell(\alpha^{-1})\tq \alpha \in\Gamma_{x,y}\}
 \\&=&
 \inf \{{}^s\ell(\alpha)\tq \alpha \in\Gamma_{x,y}\}
 \\&=&
\delta_{{}^s\ell}(x,y).
\end{eqnarray*}
\end{proof}

%____________________
\section{Weak Finsler structures}

In the paper \cite{PT2}, we introduced the following definition: 

\begin{definition}
Let $M$ be a $C^1$ manifold and let $TM$ be its tangent bundle. A \emph{weak  Finsler structure} on $M$ is a subset $\widetilde{\Omega}\subset TM$ such that for each $x$ in $M$, the subset $\Omega_x=\widetilde{\Omega}\cap T_xM$ of the tangent space $T_xM$ of $M$ at $x$ is convex and contains the origin.
\end{definition}

We refer to the paper \cite{PT2} for a list of examples. 

\begin{definition}
The \emph{Lagrangian} of a weak Finlser structure $\widetilde{\Omega}$ on a $C^1$ manifold $M$ is the function on the tangent bundle $TM$ defined by
\[p(x,\xi)=p_{\widetilde{\Omega}}(x,\xi)=\inf\{ t > 0 \tq  t^{-1}\xi\in\Omega_x\}.\]
\end{definition}

We use the same letter $p$ to denote the Lagrangian and the Minkowski norm; this will be justified below (see Remark \ref{prop:s}).

We shall say that the weak Finsler structure $\widetilde{\Omega}$ is \emph{smooth} if $p$ is smooth on the complement of the
zero section of $TM$.

Let $M$ be a $C^1$ manifold equipped with a weak  Finlser structure $\widetilde{\Omega}$ and with Lagrangian $p$. There is an associated weak length structure on $M$, defined by taking $\Gamma$ to be the semigroupoid of piecewise $C^1$ paths, and defining, for each $\gamma:[a,b]\to M$ in $\Gamma$,
\begin{equation}\label{Lebesgue} 
  \ell(\gamma)=\int_a^b p(\gamma(t),\stackrel{\cdot}\gamma (t)) dt.
\end{equation}
It is proved in \cite{PT2} that the function $p : TM \to [0,\infty]$ is Borel-measurable, hence  the integral in (\ref{Lebesgue}) is well-defined.

%_______________________
\section{Symmetrization of convex sets}\label{s:sym}

We just observed that a weak  Finlser structure on a Manifold $M$ defines a weak length structure
on that manifold by Formula (\ref{Lebesgue}). We want to understand  the   symmetrization of this
 weak length structure. This question can first be addressed at the level of convex geometry as follows: given a convex 
set $\Omega$ in $\mathbb{R}^n$, define a symmetrization of $\Omega$ which is natural and which is useful in Finsler geometry. In this section, we define such a notion.

We start by recalling a few notions in convex geometry that will be used in the sequel.

\begin{definition}\label{def:radial}
Let $\Omega\subset \mathbb{R}^n$ be a (not necessarily open)  convex set and let $x$ be a point in $\overline{\Omega}$.
The \emph{radial function of $\Omega$ with respect to $x$} is the function $r_{\Omega,x}:\mathbb{R}^n\to \mathbb{R}\cup\{\infty\}$ defined by
\[
r_{\Omega,x}(\xi)=\sup\{t\in\mathbb{R}\ \vert \ (x+t\xi)\in \Omega\}.\]
\end{definition}

\begin{definition}\label{def:Minkowski} The \emph{Minkowski function of $\Omega$ with respect to $x$} is the function  $p_{\Omega,x}:\mathbb{R}^n\to \mathbb{R}\cup\{\infty\}$ defined by
\[
p_{\Omega,x}(\xi)=\frac{1}{r_{\Omega,x}(\xi)}.\]
\end{definition}

Note that in \S\ref{sec:Min}, we already considered the function $p_{\Omega,x}$ with $x=0$.
The following proposition gives a few basic properties of the Minkowski function.
\begin{proposition}\label{prop:min}
Let $\Omega$ be a convex subset of $\mathbb{R}^n$. For every $x$ in $\Omega$ and for every $\xi$ and $\eta$ in $\mathbb{R}^n$, we have

\begin{enumerate}
\item $p_{\Omega,x}(\xi)= \inf\{t\geq 0\ \vert \ \xi\in t(\Omega-x)\}$;
\item if the ray $\{x+t\xi\ \vert \ t\geq 0\}$ is contained in $\Omega$, then $p_{\Omega,x}(\xi)= 0$; 
\item $p_{\Omega,x}(\lambda\xi)= \lambda p_{\Omega,x}(\xi)$ for all $\lambda\geq 0$;
\item $p_{\Omega,x}(\xi+\eta)\leq p_{\Omega,x}(\xi) + p_{\Omega,x}(\eta)$;
\item the Minkowski function $p_{\Omega,x}$ is convex;
\item  if $x$ is in $\interieur{\Omega}$, then $p_{{\Omega},x}$ is continuous;
\item if ${\Omega}$ is closed, then ${\Omega}=\{y=x+\xi \vert \ p_{{\Omega},x}(\xi)\leq 1\}$.
\end{enumerate}
\end{proposition}

The proof is contained in \cite{PT3}.

We can give explicit formulas for the Minkowski  function $p_{{\Omega},x}$ in various cases. For instance, the Minkowski function of the closed  ball $B=B(0,R)$ in $\mathbb{R}^n$ of radius $R$ and center $0$ with respect to any point $x$ in $B$ is given by 
\[p_{B,x}(\xi) = \frac{\sqrt{\langle \xi,x\rangle^2+ (R^{2}-\vert x\vert^2)\vert \xi\vert^2}+\langle \xi,x\rangle}{(R^{2}-\vert x\vert^2)}.
\]
The Minkowski function of a half-space $ H= \{ x \in \mathbb{R}^n \tq \langle \nu , x \rangle  \leq s\}$,
where $\nu$ is a vector in $\mathbb{R}^n$ (which is orthogonal to the hyperplane bounding $H$) and where $s$ is a real number,
with respect to a point $x$ in $H$, is given by
\[  p_{H,x}(\xi) = 
  \max\left( \frac{\langle \nu ,\xi  \rangle}{s-\langle \nu , x \rangle},0\right).
\]
 The computations are made in \cite{PT3}.
 
 \medskip
 
We start by explaining what is the  symmetrization  of a convex set in the special
case where this set is a segment in $\mathbb{R}$.
\begin{definition}[Harmonic symmetrization of a segment] \label{def:h}
We first consider compact segments. Let $[a_1,a_2]$   be a compact segment in $\mathbb{R}$ and let $x$ be a point in $[a_1,a_2]$.
The  \emph{harmonic symmetrization} of $[a_1,a_2]$  with respect to $x$  is the segment $[b_1,b_2]$ defined by the following two properties:
\begin{enumerate}
\item $x$ is the center of $[b_1,b_2]$;
\item \label{2} $\displaystyle \frac{1}{\vert b_1-x\vert}=\frac{1}{2}\left( 
\frac{1}{\vert a_1-x\vert}+ \frac{1}{\vert a_2-x\vert}\right)$;
\item \label{3} $(a_2-a_1)\in\mathbb{R}_+(b_2-b_1)$.
\end{enumerate}
In words, the definition says that  $[b_1,b_2]$ is the  harmonic symmetrization at $x$ of $[a_1,a_2]$ if
$[b_1,b_2]$ is centered at $x$ and if its half-length is the harmonic mean of $|a_1-x|$ and $|a_2-x|$.

We then define the harmonic symmetrization of an open segment $(a_1,a_2)$ as the interior of the harmonic symmetrization of the closure $[a_1,a_2]$ of $(a_1,a_2)$. We can likewise define harmonic symmetrizations of half-open intervals. The harmonic symmetrization of a half-open interval $[a_1,a_2)$ is a half-open interval $[b_1,b_2)$ such that the closed interval $[b_1,b_2]$ is the harmonic symmetrization of the closed interval $[a_1,a_2]$. (Note that the harmonic symmetrization of a half-open interval is not symmetric. The notion of harmonic symmetrization is well-behaved for open and closed convex sets.)

Next, we define the  harmonic symmetrization of an unbounded segment in $\mathbb{R}$ by extending the above definition by continuity.  More precisely, if $a_2$ but not $a_1$ is at infinity, e.g. if $[a_1,a_2]$ is an infinite ray $[a_1,\infty)$, then, extending by continuity the values given by Equation (\ref{2}) above, the value $\vert a_2-x\vert$
 is infinite, the value  $\vert a_1-x\vert $ is finite, and therefore the value $\vert b_1 - x\vert$ is finite. In particular,  the harmonic  symmetrization of an infinite ray  with respect to a point on that ray is a bounded segment. An analogous definition holds when $a_1$ is at infinity, and not $a_2$.

Finally, by extending continuously the values given by Equation (\ref{2}), the harmonic symmetrization of the whole real line is the real line itself.
\end{definition}

We shall use the unified notation 
\[
 I_2=\mathcal{H}(I_1,x)
\] 
to denote the fact that the interval $I_2$ is the harmonic symmetrization of the interval $I_1$
with respect to $x$.

\medskip

Let us now consider an arbitrary convex set ${\Omega}$ in $\mathbb{R}^n$. For any point $x$ in ${\Omega}$ and for any  non-zero vector $\xi$ in $\mathbb{R}^n$,  the  \emph{section of ${\Omega}$ through $x$ in the direction $\xi$} is the interval $S_{{\Omega},x}(\xi) = (x+\mathbb{R}\xi )\cap {\Omega}$.

\begin{definition}[Harmonic symmetrization of a convex set]\label{def:harsym}
 Let ${\Omega}$ be a convex subset of $\mathbb{R}^n$ and let $x\in {\Omega}$.
The  \emph{harmonic symmetrization of ${\Omega}$  centered at $x$} is the set $\mathcal{H}({\Omega},x)$ obtained  by replacing each section of ${\Omega}$ through $x$ by its harmonic symmetrization with respect to $x$. In other words, we have
\begin{equation*} \label{eqdef.harm}
 \mathcal{H}({\Omega},x) = \bigcup_{\xi \in \mathbb{S}^{n-1}}  \mathcal{H}( S_{{\Omega},x}(\xi) , x ).
\end{equation*}
\end{definition}

The Minkowski function of  $\mathcal{H}({\Omega},x)$ with respect to $x$ will be denoted by  $q_{\Omega,x}$, that is $$q_{\Omega,x} = p_{\mathcal{H}({\Omega},x),x}.$$
 The following results then follow directly from the definitions:
\begin{proposition}\label{prop.minkharsym}
Let ${\Omega}$ be a convex subset of $\mathbb{R}^n$ and let $x$ be an element of $\Omega$. Then, 
 the Minkowski function of  $\mathcal{H}=\mathcal{H}({\Omega},x)$ with respect to $x$ is given by
$$q_{\Omega,x}(\xi)=\frac{1}{2}(p_{\Omega,x}(\xi)+p_{\Omega,x}(-\xi)).$$
In particular, we have 
$$ \mathcal{H}({\Omega},x)=\{y\in\mathbb{R}^n\ \vert \ \frac{1}{2}\left( p_{{\Omega},x}(y-x)+p_{{\Omega},x}(x-y)\right)\leq 1\}$$
if $\Omega$ is closed, and
$$ \mathcal{H}({\Omega},x)=\{y\in\mathbb{R}^n\ \vert \ \frac{1}{2}\left( p_{{\Omega},x}(y-x)+p_{{\Omega},x}(x-y)\right) < 1\}$$
if $\Omega$ is open.
\end{proposition}
 
The following are basic properties of harmonic symmetrization,  they are proved in \cite{PT3}.

\begin{proposition}\label{prop:har} Let ${\Omega}$ be a convex subset of $\mathbb{R}^n$ and let $x$ be an element of $\Omega$. Then, 
\begin{enumerate}

\item if $\Omega$ is open (respectively closed) then $ \mathcal{H}({\Omega},x)$ is open (respectively closed);

\item the closure of $\mathcal{H}({\Omega},x)$ is symmetric with respect to $x$;

\item $\mathcal{H}({\Omega},x)$ is convex;

\item  if $\Omega$ is closed or open, then $\mathcal{H}({\Omega},x)={\Omega}$ if and only if $\Omega$ is symmetric with respect to $x$;

\item   the restriction of the map $({\Omega},x)\mapsto \mathcal{H}({\Omega},x)$ to the set of closed and bounded pointed convex sets $(\Omega,x)$ is continuous with respect to the Hausdorff topology;

\item the assignment $({\Omega},x)\mapsto \mathcal{H}({\Omega},x)$ is equivariant with respect to affine transformations;

\item If ${\Omega}$ is a polyhedron, then so is $ \mathcal{H}({\Omega},x)$;

\item If ${\Omega}$ is bounded by a quadric, then $ \mathcal{H}({\Omega},x)$ is bounded by an ellipsoid.
\end{enumerate}
\end{proposition}
 
 The harmonic symmetrization $\mathcal{H}({\Omega},x)$ is computable in a certain number of cases. For instance, one can give a formula for the harmonic symmetrization of a closed unit ball with respect to an arbitrary point. There also exist formulas for the harmonic symmetrization based on the notion of polar dual of a convex set. Details are given in \cite{PT3}.

 We now return to the question of symmetrization of a weak Finsler structure.

\section{Harmonic symmetrization of a weak Finsler structure}

A weak Finsler structure is a field of convex sets in the tangent bundle of a differentiable manifolds. Its 
harmonic symmetrization is naturally defined as the field of harmonic symmetrizations of each of these convex sets:

\begin{definition}[Harmonic symmetrization of a weak Finsler structure] Let  $M$ be a $C^1$ manifold equipped with a weak Finsler structure $\widetilde{\Omega}=\cup_{x\in M}\widetilde{\Omega}_x\subset TM$.
The \emph{harmonic symmetrization} of $\widetilde{\Omega}$  is the weak Finsler structure $\mathcal{H}(\widetilde{\Omega})\subset TM$ defined as 
\begin{equation}\label{eq:F}
\displaystyle \mathcal{H}(\widetilde{\Omega})=\cup_{x\in M}\mathcal{H}(\widetilde{\Omega}_x,0).
\end{equation}
\end{definition}
In other words, $\mathcal{H}(\widetilde{\Omega})$ is the Finsler structure obtained by taking in each
tangent space $T_xM$ the harmonic symmetrization of the convex set
$\widetilde{\Omega}_x$ with respect to the origin $0$ of $T_xM$.

Using Proposition \ref{prop.minkharsym}, we see that  the Lagrangian of $\mathcal{H}(\widetilde{\Omega})$ is given by
\begin{equation}\label{eq.lagrharm}
  q(x,\xi) =  \frac{1}{2}\left( p(x,\xi)+ p(x,-\xi)\right),
\end{equation}
where  $p=p_{\widetilde{\Omega}}$ is the the Lagrangian of $\widetilde{\Omega}$. If $\widetilde{\Omega}$ is open in $TM$, we then  have
\[
  \mathcal{H}(\widetilde{\Omega})=\{(x,\xi)\in TM \ \vert \  q(x,\xi) = \frac{1}{2}\left( p(x,\xi)+ p(x,-\xi)\right) <1\}.
\]

\begin{theorem}\label{th:1} Let $M$ be a $C^1$ manifold and let $\widetilde{\Omega}$ be a weak Finsler structure on $M$. Then, we have the following:
\begin{enumerate}

\item \label{l1} The arithmetic symmetrization of the length structure $\ell_{\widetilde{\Omega}}$ associated to $\widetilde{\Omega}$ is the length structure $\ell_{\mathcal{H}(\widetilde{\Omega})}$
associated to the harmonic symmetrization $ \mathcal{H}(\widetilde{\Omega})$ of $\widetilde{\Omega}$.

\item \label{l2} Suppose that for every $x$ and $y$ in $M$ there exists a bi-minimal path joining $x$ and $y$. Then, the distance associated to the harmonic symmetrization $\mathcal{H}(\widetilde{\Omega})$ is the arithmetic symmetrization of the distance $d_{\widetilde{\Omega}}$, that is:
 \[d_{\mathcal{H}(\widetilde{\Omega})}=\frac{1}{2}(d_{\widetilde{\Omega}}(x,y)+ d_{\widetilde{\Omega}}(y,x)).\]
\end{enumerate}
\end{theorem}

\begin{proof}
We first prove (\ref{l1}). Let $p$ be the Lagrangian of $\widetilde{\Omega}$ and let $q$ be the Lagrangian of $ \mathcal{H}(\widetilde{\Omega})$.  
Using Formula (\ref{eq.lagrharm}),  the length of an arbitrary piecewise $C^1$ path $\alpha:[0,1]\to M$
can be computed as follow:
\begin{eqnarray*}
\ell_{\mathcal{H}(\widetilde{\Omega})}(\alpha)&=& 
\int_0^1q(\alpha(t),\dot{\alpha}(t))dt\\
&=&\frac{1}{2}\left(\int_0^1p(\alpha(t),\dot{\alpha}(t))dt+
\int_0^1p(\alpha(t),-\dot{\alpha}(t))dt\right)\\
&=&\frac{1}{2}\left(\int_0^1p(\alpha(t),\dot{\alpha}(t))dt+\int_0^1p(\alpha(1-t),-\dot{\alpha}(1-t))dt\right)\\
&=&\frac{1}{2}\left(\ell_{\widetilde{\Omega}}(\alpha)+\ell_{\widetilde{\Omega}}(\alpha^{-1})\right)\}.
\end{eqnarray*}
This proves Property (\ref{l1}).

 Property  (\ref{l2}) follows from Proposition \ref{prop:minimizing}
 \end{proof}

%_____________________
\section{The tautological weak Finsler structure and the Funk weak metric}
\label{TWF}
 
In this section, $\Omega$ is an open
convex subset of $\mathbb{R}^n$. We shall use  the natural identification $T\Omega\simeq \Omega\times \mathbb{R}^n$. 

\begin{definition}[The tautological weak Finsler structure]  
The \emph{tautological weak Finsler structure} on $\Omega$  is the weak Finsler structure $\widetilde{\Omega}\subset T\Omega$ defined by
 $$
 \widetilde{\Omega}=\{(x,\xi)\in T\Omega \tq x\inÊ\Omega\text{ and } x+\xi\in\Omega\}.
 $$
\end{definition}
This structure is termed as  ``tautological" because the fiber over each point $x$ of $\Omega$ is the set $\Omega$ itself, with the origin at $x$.

The following is a consequence of the definitions, and it is proved in \cite{PT2}.
  
\begin{remark}\label{prop:s} Let $\Omega$ be an open
convex subset of $\mathbb{R}^n$ equipped with its tautological weak Finsler structure 
 $\widetilde{\Omega}$. Then, for every $x$ in $\Omega$, the Lagrangian of any tangent vector $\xi$ at $x$ is given by
$p_{\Omega,x}(\xi)$, where $p_{\Omega,x}$ is the Minkowski function of $\Omega$ with respect to $x$.
\end{remark}
 
Given an open convex subset $\Omega$ of $\mathbb{R}^n$, we denote by $d_{\Omega}$ the weak length metric associated to the tautological weak Finsler metric on $\Omega$, as defined in \S\ref{sec:weak}.
Recalling Definition \ref{def:Funk} of  the Funk weak metric, we have the following:

 \begin{theorem}\label{th.TotFunk}
Let $\Omega$ be an open convex subset of $\mathbb{R}^n$ equipped with its tautological weak Finsler structure. Then, for every $x$ and $y$ in $\Omega$, the Euclidean segment connecting $x$ and $y$ is of minimal length, and the weak metric on $\Omega$ associated to the tautological weak Finsler structure is the Funk weak metric:
$$d_{\Omega}(x,y) = F(x,y).$$
\end{theorem}

\begin{proof}
We give a sketch of the proof. Details are contained in  \cite{PT2}. 
 Let us fix two points $x$ and $y$ in $\Omega$ and let $\gamma : [0,1] \to \Omega$ be the affine segment from $x$ to $y$. Recall that  $R(x,y)$ denotes the ray with origin $x$ and parallel to the vector $\xi = (y-x)$.

We first consider the case where $R(x,y) \subset \Omega$. In this case, we have $p_{\Omega,z}(\xi)=0$ for any point $z$ on the ray $R(x,y)$ , and therefore
$$
 d_{\Omega}(x,y) \leq \ell (\gamma) =\int_0^1 p_{\Omega,\gamma (t)} (\dot \gamma (t)) dt = 0.
$$
Thus, in this case, $d_{\Omega}(x,y) = F(x,y)=0$.
 
\medskip

We next assume that $R(x,y) \not\subset \Omega$ and set $a^+ = R(x,y) \cap \partial\Omega$. A direct computation shows in that case that 
$$
 \ell (\gamma) = \log\frac{|x-a^+|}{|y-a^+|} = F(x,y)
$$
(see \cite{PT2}). Therefore, $d_{\Omega}(x,y) \leq F(x,y)$.

 It remains to prove the converse inequality $d_{\Omega}(x,y) \geq F(x,y)$. This is done in two steps:
 
 $\bullet$ We first consider the case where $\Omega = U$ is a half space. In that case the Lagrangian is explicitly computable and one checks directly that any  (piecewise) $C^1$ curve in
$U$ joining $x$ to $y$ has length at most   $\log\frac{|x-a^+|}{|y-a^+|}$. This implies that $d_{U}(x,y) = F(x,y)$, in the case where $U$ is a half space.

 $\bullet$ We conclude by a monotonicity argument. It is easy to check that if $\Omega_1\subset \Omega_2$, then $d_{\Omega_1}(x,y)\geq d_{\Omega_2}(x,y)$.  Let us choose a half space 
 $U\subset \mathbb{R}^n$ bounded by a support hyperplane of $\Omega$ at the point $a^+$, i..e such that $a^+\in \partial U$ and $U\supset \Omega$. Then 
 $$
   d_{\Omega}(x,y)\geq d_{U}(x,y) = \log\frac{|x-a^+|}{|y-a^+|} = F(x,y).
 $$
\end{proof}

\medskip

Note that the triangle inequality for the Funk weak metric is now an obvious consequence of  Theorem \ref{th.TotFunk}.   
 
%_________________________
\section{The reversible tautological structure and the Hilbert metric}\label{s:t}

\begin{definition}[The reversible tautological weak Finsler structure]  Let $\Omega$ be an open
convex subset of $\mathbb{R}^n$. The \emph{reversible tautological weak Finsler structure} on $\Omega$  is the harmonic symmetrization of the tautological weak Finlser structure of $\Omega$.
\end{definition}

 In other words, the reversible tautological weak Finsler structure on $\Omega$ is the weak Finsler structure
 given by 
  \[\widetilde{\mathcal{H}(\Omega)}=\displaystyle \cup_{x\in \Omega}\widetilde{ \mathcal{H}(\Omega_x)}\]
where for each $x$  in $\Omega$, the set $\widetilde{ \mathcal{H}(\Omega_x)}\subset T_x\Omega$ is 
the harmonic symmetrization with respect to the origin of the convex open set $\widetilde{\Omega}_x$.

The use of the term ``reversible" will be justified in Theorem \ref{th:reversible} at the end of this section.

 \begin{proposition} Let $\Omega$ be an open
convex subset of $\mathbb{R}^n$ equipped with the reversible tautological weak Finlser structure. Then, the norm $q_{\Omega}(x,\xi)$ of each tangent vector $\xi$ to $\Omega$ at $x$ is given by the formula
 \[
 q_{\Omega}(x,\xi)= \frac{1}{2} (p_{\Omega}(x,\xi)+  p_{\Omega}(x,-\xi)).
 \]
 \end{proposition}

\begin{proof}
This follows from equation (\ref{eq.lagrharm}).
\end{proof}
 
 We already recalled, in the introduction of this paper, the definition of Hilbert metric for a bounded convex domain. For a more general convex domain, the definition is somehow more cumbersome, and the idea is to extend the formula by continuity. More precisely, we give the following:

\begin{definition}[The Hilbert metric] 
Let $\Omega$ be an open
convex subset of $\mathbb{R}^n$. The \emph{ Hilbert metric} of $\Omega$  is the metric  on $\Omega$ denoted by $H_{\Omega}$ and defined, for $x$ and $y$ in $\Omega$, by the formula 
 \[\displaystyle H_{\Omega}(x,y)=
\begin{cases}
\displaystyle \frac{1}{2} \log \left( \frac{\vert x-a^+\vert}{\vert y-a^+\vert}  \frac{\vert x-a^-\vert}{\vert y-a^-\vert}  \right) & \  \text{ if } \ x\not= y,  R(x,y)\not\subset \Omega \ \text{ and } \ R(y,x)\not\subset \Omega\\
\displaystyle \log\frac{\vert x-a^+\vert}{\vert y-a^+\vert} = F(x,y)&\  \text{ if } \ x\not= y, 
R(x,y)\not\subset \Omega\ \text{ and }\  R(y,x)\subset \Omega\\
\displaystyle \log\frac{\vert x-a^-\vert}{\vert y-a^-\vert} = F(y,x)&\  \text{ if } \ x\not= y, 
R(x,y)\subset \Omega\ \text{ and } \ R(y,x)\not\subset \Omega\\
\displaystyle 0 &\  \mathrm{ otherwise}.
\end{cases}
\]
\end{definition}
 
Observe that the Hilbert metric of $\Omega$ is the arithmetic symmetrization of the Funk weak metric of $\Omega$,  namely, we have
\begin{equation}\label{form:Hilbert}
 \displaystyle H_{\Omega}(x,y)=\frac{1}{2}\left( F_{\Omega}(x,y)+ F_{\Omega}(y,x)\right).
 \end{equation}

  \begin{theorem}\label{th:reversible} Let $\Omega$ be an open convex subset of $\mathbb{R}^n$.  The distance function associated to the reversible tautological weak Finsler structure on $\Omega$  is the Hilbert distance.
Furthermore, the affine segments in $\Omega$ are minimal paths for the Hilbert metric.  
 \end{theorem}
 
\begin{proof} We use the fact that the Hilbert metric on $\Omega$ is the arithmetic symmetrization of the Funk weak metric of $\Omega$. By Item (1) in Theorem \ref{th:1}, the arithmetic symmetrization of the length structure associated to the tautological weak Finsler structure on $\Omega$ is the length function associated to the reversible tautological weak Finsler structure on $\Omega$. 
By Theorem \ref{th.TotFunk}, the weak metric associated to the tautological weak Finsler structure is the Funk weak metric $F_{\Omega}$. Theorem \ref{th.TotFunk} also says that the Euclidean paths in $\Omega$ are bi-minimal paths for the Funk weak metric. Using this fact, Proposition \ref{prop:minimizing} implies that the weak metric associated to the reversible tautological weak Finsler structure is the arithmetic symmetrization of the weak metric associated to the tautological weak Finsler structure, and this gives the desired result. 

The fact that  the affine segments  are minimal paths follows from the corresponding fact for the Funk weak metric
(see \cite{PT2}).
\end{proof}
 
As announced in the introduction, this directly shows that the Hilbert metric comes from a (weak)  Finsler structure, with no smoothness assumption. The triangle inequality is then a consequence of this fact and needs no ad-hoc proof. 
   
Finally, let us note that for convenience, we assumed throughout this paper that our convex sets are subsets of $\mathbb{R}^n$, but our results and their proofs are valid in any real affine finite- or infinite-dimensional Banach vector space.

%__________________


\begin{thebibliography}{99}

\bibitem{Benoist1} Y. Benoist,  {Convexes divisibles I}, in: Dani, S. G. (ed.) et al., Algebraic groups and arithmetic,  Proceedings of the international conference, Mumbai, India,  2001, New Delhi, Narosa Publishing House/Published for the Tata Institute of Fundamental Research, 339--374 (2004).

\bibitem{Benoist2} Y. Benoist,  {Convexes divisibles II}, 
Duke Math. J. \textbf{ 120}, No. 1, 97-120 (2003).

\bibitem{Benoist3} Y. Benoist, {Convexes divisibles III}, Ann. Sci. Ec. Norm. Sup\'er. (4) \textbf{ 38}, No. 5, 793-832 (2005).

 \bibitem{Benoist4} Y. Benoist,  
{Convexes divisibles IV}: Structure du bord en dimension 3, Invent. Math. \textbf{ 164}, No. 2, 249-278 (2006).

  \bibitem{Busemann1942} H. Busemann, \emph{Metric methods in Finsler spaces and in the foundations of geometry}, 
  Annals of Mathematics Studies \textbf{8},  Princeton University Press  (1942).

  \bibitem{Busemann1944} H. Busemann, \emph{Local metric geometry},  Trans. Amer. Math. Soc.  \textbf{56},  (1944) 200--274.

  \bibitem{Busemann1955} H. Busemann,   \emph{The geometry of geodesics}, Academic Press  (1955), reprinted by Dover in 2005.

\bibitem{Busemann1970} H. Busemann,
 \emph{Recent synthetic differential geometry,}
Ergebnisse der Mathematik und ihrer Grenzgebiete, \textbf{54},  Springer-Verlag, 1970.
    

\bibitem{BusemannMayer}  H. Busemann \& W. Mayer,  \emph{On the Foundations of Calculus of Variations},  
Trans. Amer. Math. Soc.  \textbf{49},  (1941) 173--198 .
 
\bibitem{ChernShen} S.S. Chern and Z. Shen, \emph{Riemann-Finsler Geometry}, Nankai Tracts in Mathematics, vol. 6 World Scientific 2005.

\bibitem{CV} B. Colbois \& P. Verovic, { Hilbert geometry for strictly convex domains}, Geom. Dedicata \textbf{105} (2004), 29--42. 

\bibitem{CVV} B. Colbois,  P. Verovic \& C. Vernicos, {Hilbert geometry for convex polygonal domains}, preprint 2008, hal-00271373

\bibitem{Eggleston} H. G. Eggleston, \emph{Convexity},  Cambridge Tracts in Mathematics and Mathematical Physics No. 47, Cambridge University Press, 1958.

\bibitem{FEN} W. Fenchel,\emph{ Convex cones, sets, and functions}, Mimeographed Notes by D. W. Blackett of Lectures at Princeton University, Spring Term, 1951, Princeton, 1953.

\bibitem{FK} T. F\"ortsch \& A. Karlsson, {Hilbert metrics and Minkowski norms}, J. Geom. \textbf{83}, 1-2 (2005), 22Ð31. 

 \bibitem{Funk} P. Funk, {\"Uber Geometrien, bei denen die Geraden die K\"urzesten sind}, Math. Ann. \textbf{101} (1929), 226--237.
   
 \bibitem{Harpe} P.  de la Harpe, { On HilbertÕs metric for simplices}, Lond. Math. Soc. Lect. Note Ser. 1, \textbf{181} (1993), 97--119.
  
   
 \bibitem{Hausdorff} F. Hausdorff,   
\emph{ Set theory}, Chelsea 1957.
 

\bibitem{Hilbert} D. Hilbert,  {Ueber die gerade Linie als k\"urzestes Verbindung zweier Punkte},
Math. Ann. XLVI. 91-96 (1895).
 
 \bibitem{Hilbert2} D. Hilbert,   \emph{ Grundlagen der Geometrie}, 
B. G. Teubner, Stuttgart 1899, several later editions revised by the author,
and several translations.
 
 \bibitem{KN} A. Karlsson \& G. A. Noskov,  
{The Hilbert metric and Gromov hyperbolicity},
Enseign. Math.,  IIe. S\'er. \textbf{48}, No. 1-2, 73-89 (2002). 

 \bibitem{Minkowski} H. Minkowski, \emph{Theorie der konvexen K\"orper, insbesondere Begr\"undung ihres Ober-fl\"achenbegriffs}, in Gesammelte Abhandlungen, Teubner, Leipzig, 1911.
 

\bibitem{PT1} A. Papadopoulos \& M. Troyanov, {Weak metrics on Euclidean domains}, JP Journal of Geometry and Topology   \textbf{7}, Issue 1 (March 2007), pp. 23-44.

\bibitem{PT2} A. Papadopoulos \& M. Troyanov,  {Weak Finsler Structures and the Funk Metric}, preprint 2008, available on arXiv:0804.0705v1.

\bibitem{PT3} A. Papadopoulos \& M. Troyanov, {Harmonic symmetrization of  convex sets and applications},  in preparation. 

\bibitem{Socie1} E. Soci\'e-M\'ethou,  \emph{Comportements asymptotiques et rigidit\'es des g\'eom\'etries de Hilbert}, PhD thesis, University of Strasbourg, 2000. 

\bibitem{Socie2}  E. Soci\'e-M\'ethou, {Behaviour of distance functions in HilbertÐFinsler geometry}, Differential Geometry and its Applications \textbf{20}, Issue 1 (2004) 1--10.

 \bibitem{Thompson} A. C. Thompson,  \emph{Minkowski geometry}. 
 Encyclopedia of Mathematics and its Applications, 63. Cambridge University Press, Cambridge, 1996.


  \bibitem{Vernicos2005} C. Vernicos, \textit{Introduction aux g\'eom\'etries de Hilbert}, 
 Sminaire de th\'eorie spectrale et gom\'etrie, \textbf{25} (2005) 145--168. Universit\'e de Grenoble.

\bibitem{Webster} R. Webster,\emph{ Convexity}, Oxford University Press, 1994. 

\bibitem{Zaustinsky} E. M. Zaustinsky, \emph{Spaces with nonsymmetric distance}, Mem. Amer. Math. Soc. No. 34, 1959.

\end{thebibliography}
\end{document}